\documentclass[11pt]{article}
\usepackage[a4paper,margin=2.7cm]{geometry}
\usepackage{amsmath,amssymb,amsthm,mathtools}
\usepackage{graphicx,booktabs,microtype,xcolor}
\usepackage{siunitx}
\usepackage{enumitem,hyperref,cleveref}
\usepackage{tikz}
\hypersetup{colorlinks=true,linkcolor=blue!45!black,citecolor=blue!45!black,urlcolor=blue!45!black}
\graphicspath{{figures/}}
\newtheorem{theorem}{Theorem}[section]
\newtheorem{proposition}[theorem]{Proposition}

\newcommand{\R}{\mathbb{R}}
\newcommand{\N}{\mathcal{N}}
\newcommand{\V}{\mathcal{V}}
\newcommand{\I}{\mathcal{I}}

\title{\Large Mapping for Approximation: A Unified View of Rescaled, Variably Scaled and Rational Kernel Methods \thanks{This work has been done within the framework of the Internationalisation Program with the Selcuk University}}
\author{Stefano De Marchi \\Department of Medicine \\University of Padova, Italy \\ and \\Department of Mathematics \\ Selcuk University\\ Konya, Turkey}\date{\today}
\begin{document}
\maketitle
\begin{abstract}
Classical approximation methods are usually improved by changing the sampling set, increasing the number of data points, or selecting a different basis. We advocate a complementary viewpoint: keep the sampled values fixed and modify the representation through a suitable mapping. This viewpoint unifies several constructions that were originally introduced for different purposes. Rescaled radial basis function interpolation maps the interpolation operator through a normalization that enforces exact reproduction of constants. Variably scaled kernels map the geometry by lifting the data to a higher-dimensional manifold determined by a scale function. Mapped bases and fake nodes map the approximation space without resampling the data, while rational kernel expansions can be interpreted as nonlinear mappings of the trial space. We develop a common notation for these mechanisms, summarize their approximation and stability properties, and explain how discontinuous and rational variants fit the same framework. Reproducible numerical experiments illustrate constant reproduction, error reduction through data-mimicking scale functions, and the suppression of the Runge phenomenon by mapped polynomial bases. The results support a general principle: mappings provide a flexible way to adapt approximation spaces to data geometry and regularity without modifying the original observations.
\end{abstract}

\section{Introduction}
Let $\Omega\subset\R^d$, let $X=\{x_1,\ldots,x_N\}\subset\Omega$ be pairwise distinct, and let $f_X=(f(x_1),\ldots,f(x_N))^T$. A classical approximation scheme combines the sampling set, a finite-dimensional trial space, and an interpolation or approximation operator. In compact notation we write
\[
 \mathcal A=(X,\V_X,\I_X),
\]
where $\V_X$ is generated by kernels, polynomials, or other basis functions. Most classical strategies modify $X$ or enlarge $\V_X$. The central idea of this paper is different: the data remain untouched, while a mapping changes the interpolant, the geometry, or the trial space.

The viewpoint originates from three families of methods developed in the recent approximation literature: the rescaled RBF interpolant~\cite{Deparis2014,DeMarchi2017,DeMarchiWendland2020}, variably scaled kernels~\cite{Bozzini2015,VSDK2020,ShapeDriven2020}, and mapped bases without resampling~\cite{FakeNodes2020,MultivariateFake2021,GRASPA2021}. The first normalizes a standard interpolant by the interpolant of the constant function. The second lifts every point $x$ to $(x,\psi(x))$ and applies a fixed kernel in the augmented space. The third composes every basis function with a map $S$, thereby creating fake nodes used only in the reconstruction. Rational RBF expansions form a fourth, nonlinear realization in which the approximant is a quotient of two kernel expansions~\cite{Buhmann2020}.

\begin{figure}[t]
\centering
\includegraphics[width=.88\linewidth]{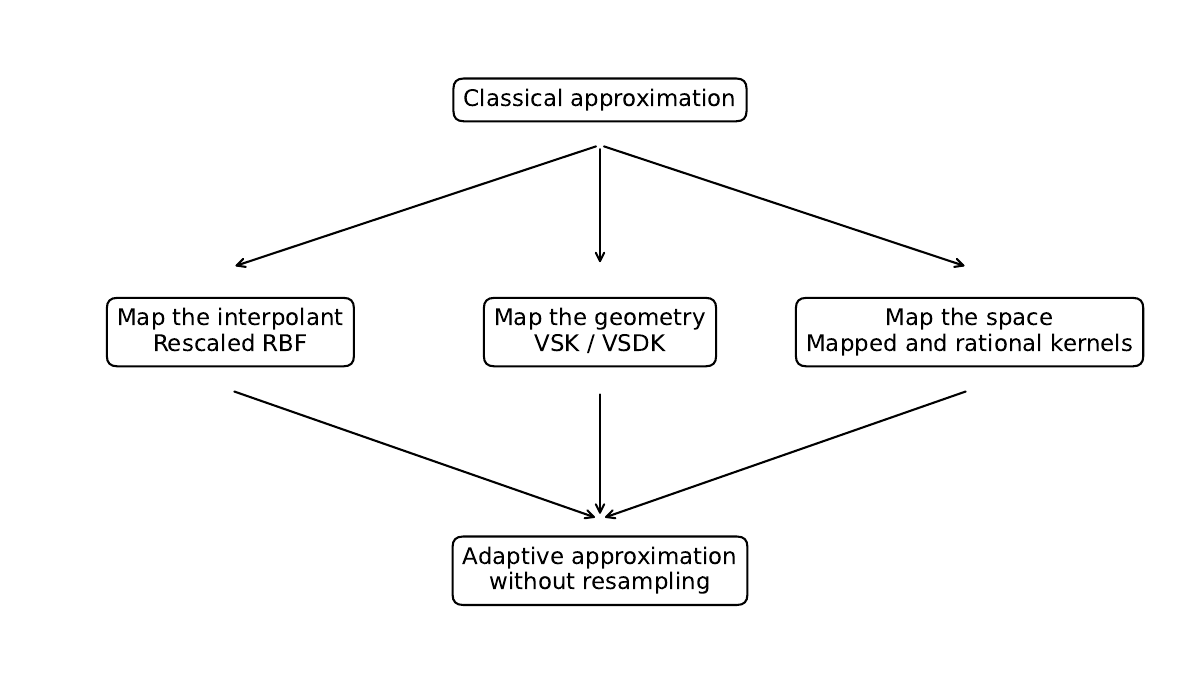}
\caption{The unifying mapping philosophy. The sampled values are preserved, while the interpolant, geometry, or approximation space is transformed.}
\label{fig:philosophy}
\end{figure}

The paper is organized as follows. Section~\ref{sec:classical} recalls kernel interpolation and the stability-accuracy trade-off. Section \ref{sec:mapping} and its subsections describe different types of mapping strategies.  Section~\ref{sec:rescaled} describes the rescaled method and its Shepard interpretation. Section~\ref{sec:vsk} treats variably scaled and discontinuous kernels. Section~\ref{sec:mapped} discusses mapped bases and fake nodes. Section~\ref{sec:rational} incorporates rational kernels into the same framework. Section~\ref{sec:unified} formulates a general mapping principle, and Section~\ref{sec:num} reports numerical experiments.

\section{Classical kernel approximation}\label{sec:classical}
In what follows, we assume that $K:\Omega\times\Omega\to\R$ is a strictly positive definite kernel. The kernel interpolation problem can be written in the form
\begin{eqnarray}\label{eq:rbf}
 P_f(x)=\sum_{j=1}^N c_jK(x,x_j),\\
  Ac=f_X, \quad A_{ij}=K(x_i,x_j).
\end{eqnarray}
For a radial kernel, $K(x,y)=\phi(\varepsilon\|x-y\|_2)$ (with $\|\cdot\|2$ the Euclidean norm), the parameter $\varepsilon>0$ controls flatness or support size of the {\it radial basis function (RBF)} $\phi$. Typical examples of RBF include Gaussian $\phi(r)=e^{-r^2}$, Mat\'ern kernels, and compactly supported Wendland functions (example, the $C^2$ Wendland function is $\phi(r)=(1-r)^4_+(4r+1)$). See, for example ~\cite{Fasshauer2007, Wendland2005, DeMP:book26},for a complete list of kernels, their contructions and properties.

The {\it separation distance} and {\it fill-distance} are two important quantities
\[
 q_X=\frac12\min_{i\ne j}\|x_i-x_j\|_2,
 \qquad h_{X,\Omega}=\sup_{x\in\Omega}\min_{x_j\in X}\|x-x_j\|_2.
\]
which play a key role for understanding stability, convergence of the interpolation, or providing error estimates and bounds for the interpolation with RBFs. 
A fundamental quantity for error estimation is the
\emph{power function}
\begin{equation}
	P_{K,X}\left(\boldsymbol{x}\right) := \sqrt{\Phi\left(\boldsymbol{x}, \boldsymbol{x}\right) -\left({b}\left(\boldsymbol{x}\right)\right)^{\intercal} \mathsf{A}^{-1} {b}\left(\boldsymbol{x}\right)},
	\label{eq:pf1}
\end{equation}
where ${b} = (K(\cdot,{x}_1), \ldots, K(\cdot,{x}_N))^{\intercal}$ and $\mathsf{A}$ is the kernel matrix defined above. Alternatively, the power function can be expressed using the cardinal functions $u_k$,
\begin{equation}
P_{K,X}(x) =\sqrt{K(x,x) - \sum_{k=1}^n K(x,x_k)\, u_k(x)} \label{eq:pf2}
\end{equation}
If a function $f$ belongs to the so-called {\it native space} $\N_K(\Omega)$, which is the Hilbert space of functions for which the kernel $K$ acts as a reproducing kernel, 
$P_{K,X}$ yields the classical pointwise estimate~\cite{Fasshauer2007}
\[
 |f(x)-P_f(x)|\le P_{K,X}(x)\,\|f\|_{\N_K(\Omega)}, \;\; x \in \Omega.
\]
Small interpolation errors often occur in flat regimes in which, on the contrary, $A$ becomes severely ill-conditioned, a manifestation of the accuracy--stability trade-off principle discussed in~\cite{Schaback1995}. Hence, shape-parameter selection reflects a trade-off between accuracy and stability. The methods we are discussing below, can be interpreted as ways of changing this trade-off by "mapping the representation rather than the observations".

\section{Mapping the interpolation operator} \label{sec:mapping}
\subsection{The rescaled method}\label{sec:rescaled}
Let $P_f$ be the kernel interpolant of $f$, and let $P_1$ denote the interpolant of the constant function $1$. Following the {\it rescaled localized RBF (RL-RBF)} construction introduced in~\cite{Deparis2014} and further analyzed in~\cite{DeMarchi2017, DeMarchiWendland2020}, the rescaled interpolant is
\begin{equation}\label{eq:rescaled}
 \widehat P_f(x)=\frac{P_f(x)}{P_1(x)},
 \qquad P_1(x)\ne0.
\end{equation}
For its construction, it uses the same nodes and the same kernel matrix as the classical kernel interpolation method, but maps the interpolation operator through a nonlinear normalization. The main properties of the rescaled interpolant can be summarized in the following Proposition.

\begin{proposition}[Interpolation and constant reproduction]
For every data vector $f_X$, $\widehat P_f(x_i)=f(x_i)$. Moreover, if $f\equiv c$, then $\widehat P_f\equiv c$ wherever $P_1$ does not vanish.
\end{proposition}
\begin{proof}
Since $P_f(x_i)=f(x_i)$ and $P_1(x_i)=1$, interpolation follows immediately. For constant data $f(x_i)=c$, uniqueness gives $P_f=cP_1$, hence $\widehat P_f=c$.
\end{proof}

Let $u_j$ be the cardinal functions, so that $P_f=\sum_jf(x_j)u_j$ and $P_1=\sum_ju_j$. Then
\begin{equation}\label{eq:shepard}
 \widehat P_f(x)=\sum_{j=1}^N f(x_j)\widehat u_j(x),
 \qquad \text{with} \qquad \widehat u_j(x)=\frac{u_j(x)}{\sum_{k=1}^Nu_k(x)}, \; \; \; \forall \, x \in \Omega.
\end{equation}
Consequently $\sum_j\widehat u_j=1$, and the method is indeed a {\it Shepard-type approximation} whose weights are the normalized cardinal functions of the kernel. This interpretation is central in the analysis of the rescaled localized RBF method as discussed in ~\cite{DeMarchi2017,DeMarchiWendland2020}.

For stability analysis, the associated Lebesgue function and constant are then defined as
\[
 \widehat\Lambda_N(x):=\sum_{j=1}^N|\widehat u_j(x)|,
 \qquad \widehat\lambda_N:=\|\widehat\Lambda_N\|_{L^\infty(\Omega)},
\]
and
\[
 \|\widehat P_f\|_{L^\infty(\Omega)}
 \le \widehat\lambda_N\|f\|_{\ell^\infty(X)}.
\]
Normalization may reduce the value assumed by the Lebesgue function, especially for compactly supported and localized bases, but the denominator must be monitored. In local or partition-of-unity implementations, the method can be applied patchwise, which avoids global denominator failures and preserves sparsity, for details see \cite{DeMarchi2017,Buhmann2020}.

A useful reinterpretation of this approach is obtained by re-writing it 
\[
 \widehat P_f(x)=\sum_{j=1}^N c_j\widehat K_j(x),
 \qquad
 \widehat K_j(x)=\frac{K(x,x_j)}{P_1(x)P_1(x_j)}.
\]
Hence, rescaling the kernel, can itself be viewed as a {\it mapping of the basis} through $\widehat{K}$.

\subsection{Mapping the geometry: variably scaled kernels}\label{sec:vsk}
Let $\psi:\Omega\subset \R^d \to\R$ be a scale function, we define the {\it lifting map}
\[
 \sigma(x)=(x,\psi(x))\in\R^{d+1}.
\]
Given a positive definite kernel $K$ on $\R^{d+1}$, the {\it variably scaled kernel (VSK)} introduced in~\cite{Bozzini2015} is
\begin{equation}\label{eq:vsk}
 K_\psi(x,y)=K(\sigma(x),\sigma(y)).
\end{equation}
The positive definiteness of $K_\psi$ is inherited by restriction to the lifted manifold $\sigma(\Omega)$. 
The following Proposition of \cite{Bozzini2015} states some additional properties of VSKs.
\begin{proposition}\label{prp:propvsk}
	Let $K$, $K_{\psi}$, and $\psi$ be as before. We have 
	\begin{enumerate}
		\item[(i)]
		If $K$ and $\psi$ are continuous, so is $K_{\psi}$. Moreover, if $\psi: \Omega \rightarrow \psi(\Omega)$ is a bijection, then $K_\psi$ inherits the positiveness properties of $K$.
		\item[(ii)]
		For the separation distances we have
		{ {\begin{equation*}
			q_X \le q_{\sigma(X)}\le (1+L)q_X
		\end{equation*}
		for any choice of the scaling map $\psi$. Indeed, by the definition of the Euclidean norm, we have immediately
        $$\|\sigma({x}_i)-\sigma({x}_j)\|^2 \ge \|{x}_i - {x}_j\|^2$$
        and also
        $$\|\sigma({x}_i)-\sigma({x}_j)\|^2=\|{x}_i-{x}_j\|^2+(\psi({x}_i)-\psi({x}_j))^2 \le \|{x}_i-{x}_j\|^2(1+L)^2,$$
        where $L$ is the Lipschitz constant of $\psi$.
       }}
		\item[(iii)]
		Let $G_{\sigma}(\Omega) = \{({x},\psi({x}))\:|\:{x} \in \Omega\} \subset \Omega \times \R$ be the graph of the scaling function $\psi$. Then, the native spaces $\mathcal{N}_{K}(G_{\sigma}(\Omega))$ and $\mathcal{N}_{K_{\sigma}}(\Omega)$ are isometrically isomorphic (cf. [Th. 2, Bozzini15]).
	\end{enumerate}
\end{proposition}
As a consequence of Proposition \ref{prp:propvsk} (ii), the variably scaled setting might improve the stability of the interpolation process by increasing the separation distance. 
The interpolation via Variably Scaled Kernels (VSKs) depends on the definition of an appropriate scaling function, but no fixed theoretical or numerical recipes for its construction have been provided.

The corresponding VSK interpolant can be defined as
\[
 P_{\psi,f,X}(x):=P_{1,f,\sigma(X)}(\sigma(x)),
\]
so the variable-scale problem in $\R^d$ is exactly a fixed-scale problem on a $d$-dimensional submanifold of $\R^{d+1}$.

For example, for the Gaussian kernel,
\[
 K_\psi(x,y)=\exp\!\left[-\varepsilon^2\bigl(\|x-y\|_2^2+|\psi(x)-\psi(y)|^2\bigr)\right].
\]
If $\psi$ mimics the data, points separated by rapid changes are farther apart in the lifted geometry, while points belonging to a common feature remain close. The asymptotic convergence order is still determined by the smoothness of the kernel and target, but the mapping can substantially change the constants and the conditioning, cf. \cite{DeMarchi2017}.

\paragraph{Discontinuous scaling.}
If the target data possesses a priori known or detected interfaces (literature exhibits many well suited techniques for edge detection, for example, the {\it Canny algorithm} \cite{Canny1986, Weibin2014}), a discontinuous label function may be used. Let $\Omega=\bigcup_{r=1}^M\Omega_r$ and set $\psi(x)=\alpha_r$ for $x\in\Omega_r$, with distinct labels $\alpha_r$. Then $K_\psi$ becomes a variably scaled discontinuous kernel (VSDK) in the sense of~\cite{VSDK2020,ShapeDriven2020}. Samples from different regions acquire an additional vertical separation $|\alpha_r-\alpha_s|$, suppressing cross-interface coupling and reducing Gibbs oscillations. This construction is especially natural in imaging, where masks or edge detectors provide the labels.

\subsection{Rational kernels as nonlinear mappings of the trial space}\label{sec:rational}
Following the eigen-rational RBF framework analyzed in~\cite{Buhmann2020}, rational RBF expansions take the general form of
\begin{equation}\label{eq:rational}
 R_f(x):=\frac{P_g(x)}{P_h(x)}
=\frac{\sum_{i=1}^N\alpha_iK(x,x_i)+\sum_{m=1}^L\gamma_mp_m(x)}
 {\sum_{k=1}^N\beta_k\overline K(x,x_k)}.
\end{equation}
The name eigen-rational comes from the fact that the denominator of this rational function is explicitly built using the eigenvector corresponding to the largest eigenvalue of the kernel matrix. Furthermore, the denominator values $h_i=P_h(x_i)$ are selected first so that the numerator interpolates $g_i=f_i h_i$. Hence $R_f(x_i)=f_i$ whenever $P_h(x_i)\ne 0$.

The construction is naturally interpreted as a nonlinear map
\[
 \mathfrak R:\V_g\times\V_h \longrightarrow \{u/v:u\in\V_g,\ v\in\V_h,\ v\ne0\},
 \qquad \mathfrak R(u,v)=u/v\,,
\]
where the spaces $\V_g, \V_h$ are those of the kernels that appear in the numerator and the denominator in \eqref{eq:rational}, respectively.

This extends the above presented rescaled method: choosing $P_h=P_1$ gives \eqref{eq:rescaled}; the relation between rescaled and rational constructions is discussed in~\cite{Buhmann2020}. The rational setting offers additional freedom for steep gradients, poles outside the domain, and adaptation of local characteristics. At the same time, denominator control and selection of $h_i$ are essential for stability.

\subsection{Mapping the approximation space: fake nodes}\label{sec:mapped}
Let $S:\Omega\to S(\Omega)$ be injective and let $\{B_1,\ldots,B_N\}$ be a basis on $S(\Omega)$. Define the mapped basis as
\[
 B_j^S=B_j\circ S,
 \qquad \V_N^S=\operatorname{span}\{B_1^S,\ldots,B_N^S\}.
\]
The mapped interpolant is then 
\begin{equation}\label{eq:mapped}
 R_f^S(x)=\sum_{j=1}^N\alpha_j^SB_j(S(x))=P_g(S(x)),
\end{equation}
where $g(S(x_i))=f(x_i)$. Notice, that the values are not resampled: the mapped locations $S(X)$ are used only in the reconstruction and are therefore called {\it fake nodes}, and this approach is known as {\it Fake Nodes Approach (FNA)}~\cite{FakeNodes2020,MultivariateFake2021}. The Proposition \ref{prop:inheritance} shows the analysis of the error and the stability of the FNA, which is inherited from the classical approach.

\begin{proposition}[Inheritance under a bijective map]  \label{prop:inheritance}
Suppose $S$ is a homeomorphism between $\Omega$ and $S(\Omega)$. Let $f=g\circ S$. Then
\[
 \|R_f^S-f\|_{L^\infty(\Omega)}
 =\|P_g-g\|_{L^\infty(S(\Omega))}.
\]
Moreover, the mapped Lebesgue constant satisfies
\[
 \Lambda^S(\Omega)=\Lambda(S(\Omega)).
\]
\end{proposition}
\begin{proof}
Both statements follow by the substitution $z=S(x)$ in the cardinal representation of $P_g$ and by the surjectivity of $S:\Omega\to S(\Omega)$.
\end{proof}

This result explains why a map that sends an "unfavorable grid" to a "favorable" one transfers the stability and approximation behavior of the latter without changing the function values. In the polynomial framework, examples include maps from equispaced nodes to Chebyshev nodes, the mapped-polynomial strategy of Adcock and Platte ~\cite{AdcockPlatte2016}, the {\it Kosloff--Tal-Ezer map}, and discontinuous maps designed to separate points across jumps \cite{FakeNodes2020}. It is worth mentioning the {\it GRASPA construction}, which combines Runge-avoiding and Gibbs-avoiding maps~\cite{GRASPA2021}. As a final remark, the FNA is very general and can be applied to any basis $B$ of functions, i.e. polynomials, piecewise polynomials, RBF, rational RBF etc.

\section{A unified mapping principle}\label{sec:unified}
We now formalize the philosophy displayed in \cref{fig:philosophy}. An approximation scheme is represented by
\[
 \mathcal A=(X,\V_X,\I_X).
\]
A mapping transformation is a triple
\[
 \mathcal M=(M_X,M_\V,M_\I)
\]
that produces the mapped triple
\[
 \widetilde{\mathcal A}
 =(M_X(X),M_\V(\V_X),M_\I(\I_X)).
\]
In this paper, the observations $f_X$ are preserved. The apparent relocation of nodes in fake-node and VSK methods occurs only internally during reconstruction.

\begin{proposition}[Mapping principle]
Let $X\subset\Omega$ be fixed and suppose a transformed operator $\widetilde\I_X$ satisfies
\[
 (\widetilde\I_X f)(x_i)=f(x_i),\qquad i=1,\ldots,N.
\]
Then the transformed scheme preserves interpolation independently of whether the transformation acts on the operator, geometry, or trial space. In particular:
\begin{enumerate}[label=(\roman*)]
\item normalization by an interpolant equal to one on $X$ gives the rescaled method;
\item pullback by an injective lifting $\sigma$ gives VSK interpolation;
\item composition of a basis with an injective map $S$ gives fake-node interpolation;
\item a quotient $P_g/P_h$ with $P_g(x_i)=f_iP_h(x_i)$ gives rational interpolation.
\end{enumerate}
\end{proposition}
\begin{proof}
Each statement follows by evaluating the transformed approximant at $x_i$. In (i), $P_1(x_i)=1$; in (ii), the fixed-scale interpolant at $\sigma(x_i)$ equals $f_i$; in (iii), $P_g(S(x_i))=g(S(x_i))=f_i$; in (iv), the interpolation condition for the numerator gives $P_g(x_i)/P_h(x_i)=f_i$.
\end{proof}

The theorem is deliberately elementary: its value is conceptual. It identifies a common invariant, namely preservation of sampled values, while approximation and stability are modified through the representation.

\section{Numerical experiments}\label{sec:num}
In the following two experiments, both plots and tables, have been generated in Python 3.12.11: interested readers can download the script {\tt main.py} at \url{https://github.com/demarchi17/MappingBases_Unified} to reproduce the results. Dense linear systems were solved with double precision. The goal is illustrative rather than definitive; parameter searches are reported explicitly.

\subsection{Rescaled Localized RBF interpolation}
We interpolate $f(x)=x^2$ on $[0,1]$ at seven equispaced points using the $C^2$ Wendland function 
\[
 \phi(r)=(1-\varepsilon r)_+^4(4\varepsilon r+1),
\]
choosing the shape parameter $\varepsilon=5$. The classical and rescaled approximants are shown in \cref{fig:rescaled}. The rescaled method reproduces constants exactly and, in this example, modestly changes the boundary behavior while preserving interpolation.
\begin{figure}[hbt!]
\centering
\begin{minipage}{.49\linewidth}\centering
\includegraphics[width=\linewidth]{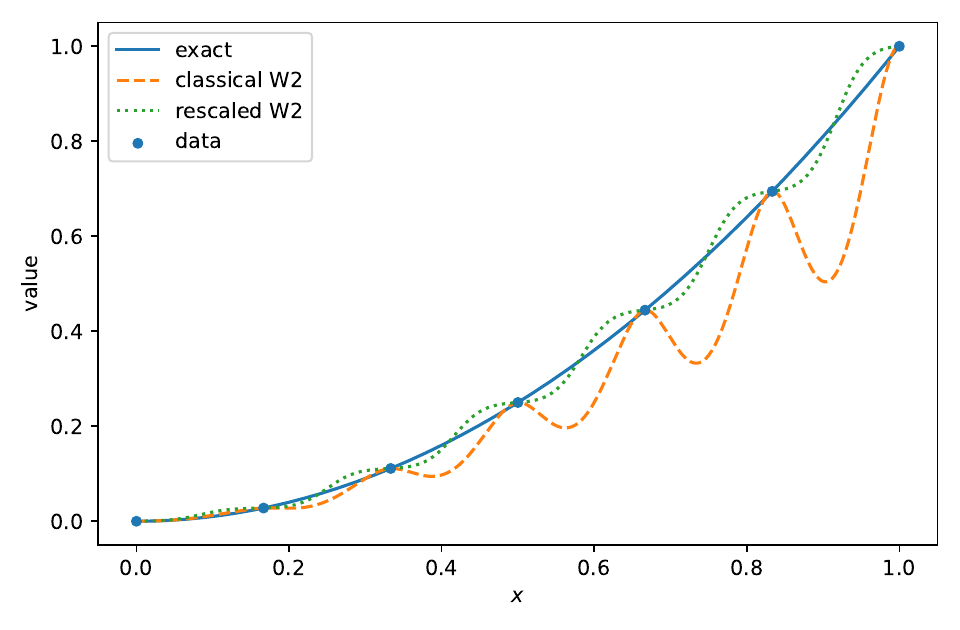}
\end{minipage}\hfill
\begin{minipage}{.49\linewidth}\centering
\includegraphics[width=\linewidth]{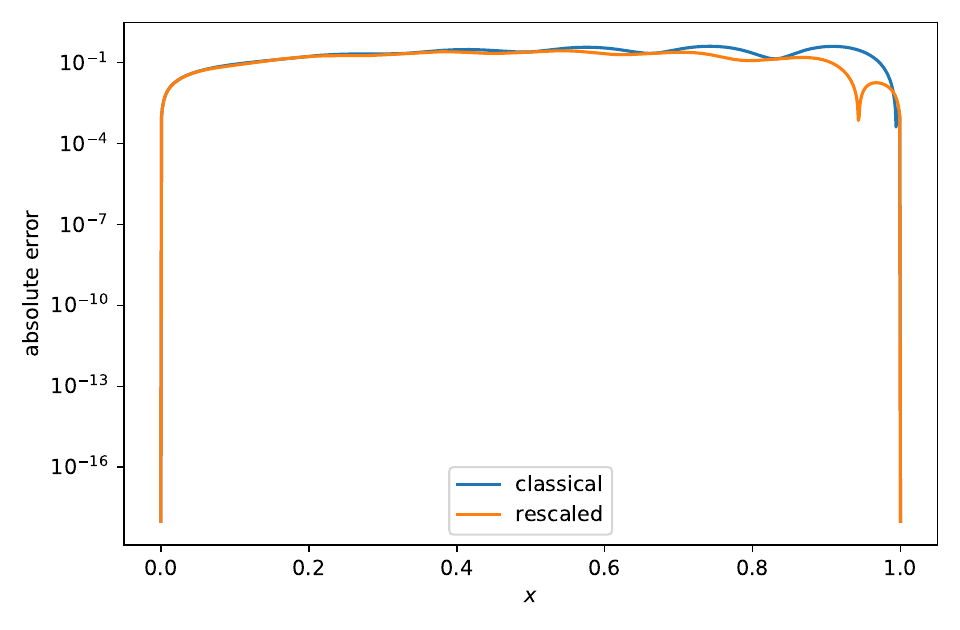}
\end{minipage}
\caption{Classical vs rescaled localized  interpolation of $f(x)=x^2$ using the $C^2$ Wendland  function (left) and pointwise errors (right).}
\label{fig:rescaled}
\end{figure}

\subsection{A data-mimicking VSK}
For the Runge function
\[
 f(x)=\frac{1}{1+25x^2},\qquad x\in[-1,1],
\]
we compare classical Gaussian interpolation with the VSK
\[
 K_\psi(x,y)=\exp\{-\varepsilon^2[(x-y)^2+(\psi(x)-\psi(y))^2]\},
 \qquad \psi(x)=\mu f(x).
\]
For each $N$, $\varepsilon$ is searched on a common logarithmic grid and $\mu\in\{0.25,0.5,1,2\}$. Since the exact target is used as the scale function, this experiment isolates the potential of a perfect data-mimicking map. In practice, $\psi$ must be estimated from the samples and regularized.

\begin{table}[hbt!]
\centering
\caption{Optimized maximum errors for Gaussian and VSK interpolation of the Runge function. The search uses the same logarithmic grid of shape parameters for both methods.}
\label{tab:vsk-runge}
\begin{tabular}{r S[scientific-notation=true] S[table-format=2.3] S[scientific-notation=true] S[table-format=2.3] S[table-format=1.2]}
\toprule
$N$ & {Gaussian error} & {$\varepsilon$} & {VSK error} & {$\varepsilon$} & {$\mu$}\\
\midrule
11 & 1.584e-03 & 6.31 & 6.704e-08 & 0.1 & 0.5 \\
17 & 1.988e-03 & 7.69 & 1.375e-08 & 0.164 & 0.25 \\
25 & 4.411e-04 & 6.31 & 1.392e-08 & 0.11 & 2 \\
33 & 8.269e-05 & 6.31 & 7.827e-09 & 0.535 & 0.5 \\
\bottomrule
\end{tabular}
\end{table}
\begin{table}[hbt!]
\centering
\caption{Maximum errors in the one-dimensional experiments.}
\label{tab:summary-errors}
\begin{tabular}{l S[scientific-notation=true] S[scientific-notation=true]}
\toprule
Experiment & {Classical} & {Mapped/rescaled}\\
\midrule
Wendland interpolation of $f(x)=x^2$ & 4.051e-01 & 2.748e-01 \\
Runge polynomial interpolation & 5.982e+01 & 3.303e-03 \\
\bottomrule
\end{tabular}
\end{table}

\begin{figure}[hbt!]
\centering
\begin{minipage}{.49\linewidth}\centering
\includegraphics[width=\linewidth]{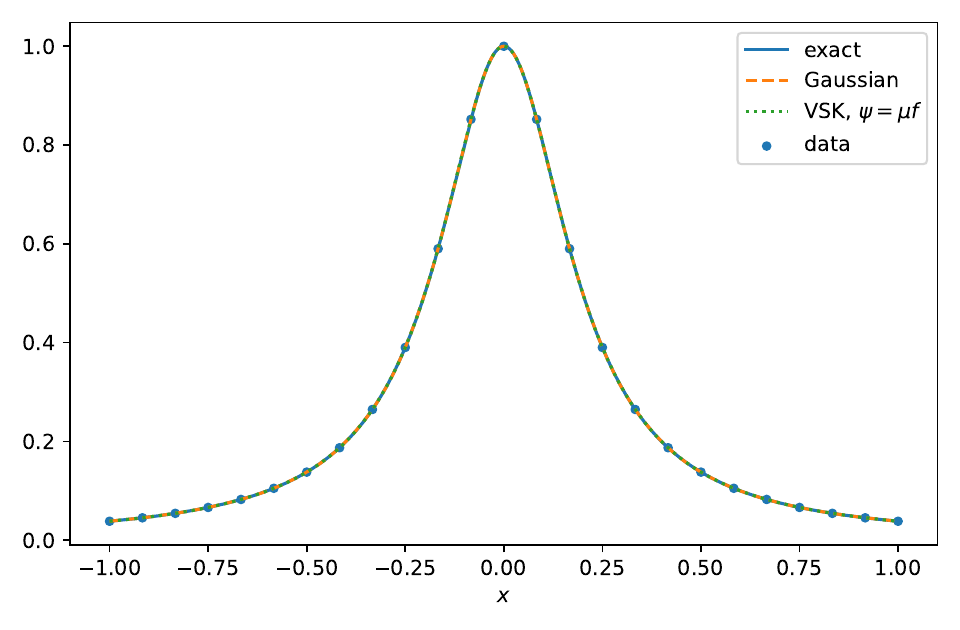}
\end{minipage}\hfill
\begin{minipage}{.49\linewidth}\centering
\includegraphics[width=\linewidth]{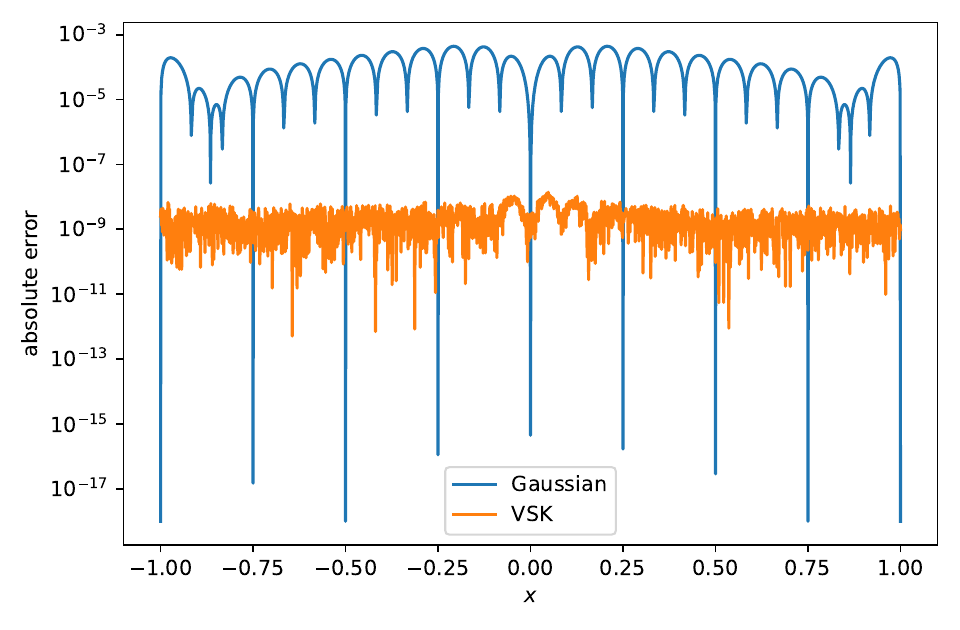}
\end{minipage}
\caption{Gaussian and VSK interpolation of the Runge function for $N=25$ (left) and corresponding errors (right). Parameters are selected by minimizing the maximum error on the evaluation grid.}
\label{fig:vsk}
\end{figure}

The VSK errors in \cref{tab:vsk-runge} are much smaller, but the selected systems can be extremely ill-conditioned. This is an important caveat: geometric adaptation does not remove the classical flat-kernel trade-off. Stable algorithms, regularization, or constrained parameter selection are required before interpreting the smallest formal errors as reliable computational gains.

\subsection{Mapped polynomial bases and fake nodes}
We interpolate the Runge function at the equispaced nodes $N=21$. The classical polynomial interpolant is compared with a mapped basis obtained from the piecewise-linear map $S$ satisfying
\[
 S(x_j)=z_j,
\]
where $z_j=-\cos( {j\pi \over N})$ are the {\it Chebyshev--Lobatto} nodes. The values $f(x_j)$ are retained and interpreted as data of a function $g$ at the fake nodes $z_j$. The mapped interpolant is $P_g(S(x))$.

\begin{figure}[t]
\centering
\begin{minipage}{.49\linewidth}\centering
\includegraphics[width=\linewidth]{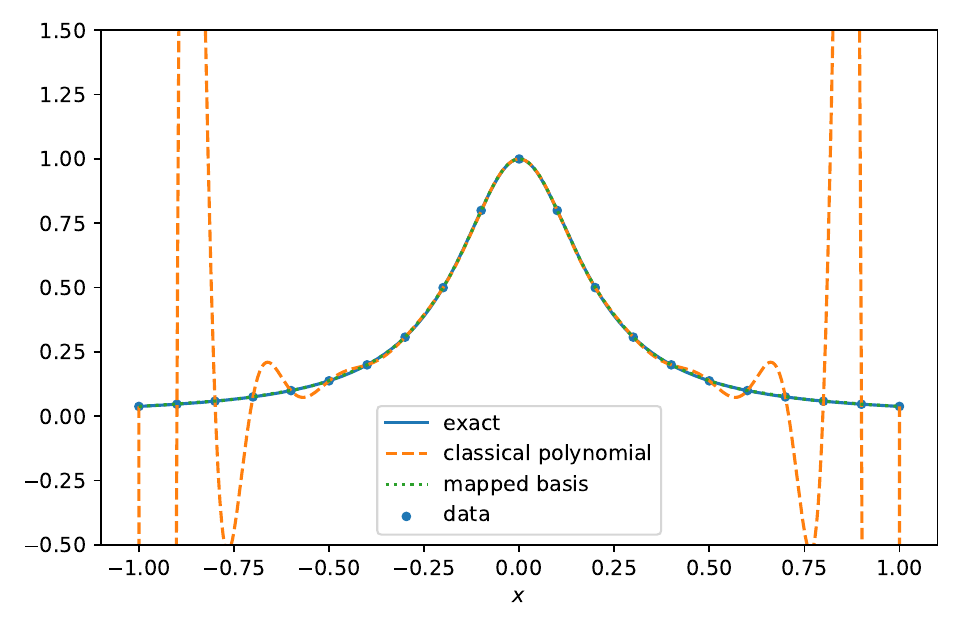}
\end{minipage}\hfill
\begin{minipage}{.49\linewidth}\centering
\includegraphics[width=\linewidth]{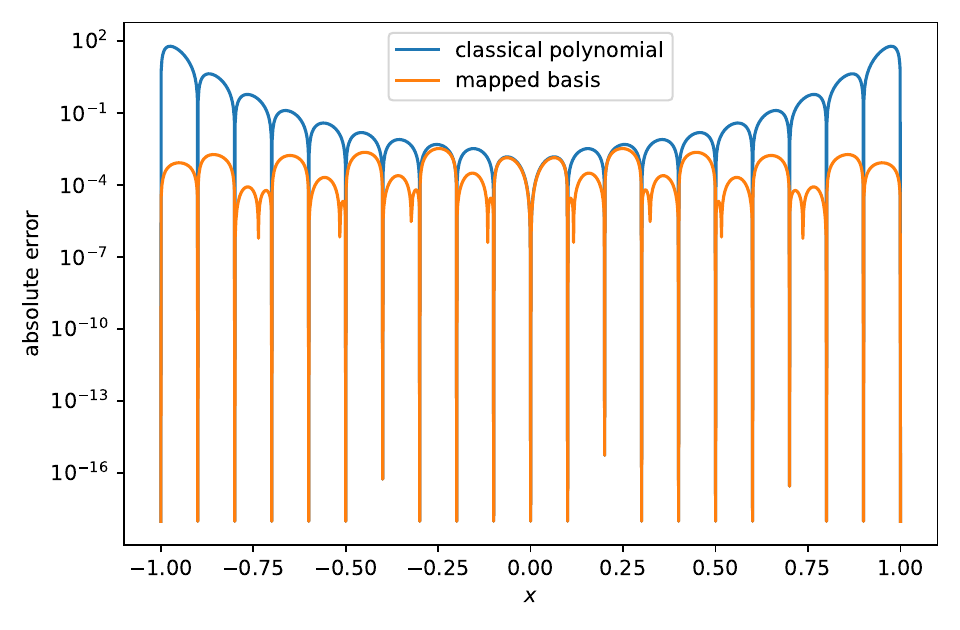}
\end{minipage}
\caption{Classical equispaced polynomial interpolation and mapped-basis interpolation of the Runge function (left), with pointwise errors (right).}
\label{fig:fake}
\end{figure}

The experiment makes the no-resampling principle explicit. The original values are never recomputed; only the basis geometry is changed. The mapped scheme inherits the favorable Lebesgue behavior of the Chebyshev grid and removes the large endpoint oscillations of the equispaced interpolant, consistently with the mapped-basis theory in~\cite{AdcockPlatte2016,FakeNodes2020}.

\paragraph{A mapped construction from equispaced nodes to Padua points.}
We now consider a two-dimensional analog of the one-dimensional
fake-Chebyshev construction on
$\Omega=[-1,1]^2$.
The space of bivariate polynomials of total degree at most $n\ge 1$, 
$\Pi_n^2$, has dimension
$
\dim \Pi_n^2
=
\frac{(n+1)(n+2)}{2}.
$
In this case a good set of nodes in $\Omega$ is the set of the {\it Padua points} \cite{Bos2006}.
A convenient representation of the {\it Padua points} of degree $n$ is as the set
\[
\mathcal P_n
=
\left\{
\left(
-\cos\frac{i\pi}{n},
-\cos\frac{j\pi}{n+1}
\right):
\begin{array}{l}
0\leq i\leq n,\\
0\leq j\leq n+1,\\
i+j\ \hbox{even}
\end{array}
\right\}.
\]
whose cardinality is exactly the dimension of $\Pi_n^2$. 

To construct a corresponding set of physical equispaced points, consider the set
$$\left\{
x_i=-1+\frac{2i}{n},\; i=0,\ldots,n,\;\qquad
y_j=-1+\frac{2j}{n+1},
\; j=0,\ldots,n+1
\right\}\,.
$$
We retain only those points satisfying the same parity condition,
\[
X_n
=
\left\{
(x_i,y_j):
0\leq i\leq n,\;
0\leq j\leq n+1,\;
i+j\ \hbox{even}
\right\}.
\]
Therefore
\[
\#X_n
=
\#\mathcal P_n
=
\dim \Pi_n^2.
\]

We define two one-dimensional piecewise-linear maps
\[
s_n:[-1,1]\to[-1,1],
\qquad
s_{n+1}:[-1,1]\to[-1,1],
\]
by imposing
\[
s_n(x_i)
=
-\cos\frac{i\pi}{n},\qquad
s_{n+1}(y_j)
=
-\cos\frac{j\pi}{n+1}.
\]

The two-dimensional map is then defined componentwise as
\[
S:\Omega\longrightarrow \Omega,
\qquad
S(x,y)
=
\left(
s_n(x),s_{n+1}(y)
\right).
\]
In particular,
\[
S(x_i,y_j)
=
\left(
-\cos\frac{i\pi}{n},
-\cos\frac{j\pi}{n+1}
\right),
\]
and therefore
\[
S(X_n)=\mathcal P_n.
\]

Let now $f:\Omega\to\mathbb R$ be the function to be approximated.
The values at the physical equispaced points are
$f_{ij}= f(x_i,y_j),\qquad i+j\ \hbox{even}.$
These values are retained but are reinterpreted as samples of a "fake" function $g$ at the Padua points, that is 
\[
g\left(
-\cos\frac{i\pi}{n},
-\cos\frac{j\pi}{n+1}
\right)
=
f(x_i,y_j).
\]

Let $p_n[g]\in\Pi_n^2$ denote the polynomial interpolant at the Padua points, 
\[
p_n[g](z_{ij})
=
g(z_{ij})
=
f(x_i,y_j).
\]

The mapped interpolant is finally defined by
\[
\boxed{
\mathcal I_n^S f(x,y)
=
p_n[g]\bigl(S(x,y)\bigr).
}
\]

At every physical interpolation point $(x_i,y_j)\in X_n$,
\[
\mathcal I_n^S f(x_i,y_j)
=
p_n[g]\bigl(S(x_i,y_j)\bigr)
=
p_n[g](z_{ij})
=
f(x_i,y_j),
\]
so that the mapped approximation preserves interpolation at the original
physical nodes.

This construction can be regarded as the two-dimensional counterpart of
the mapped fake-Chebyshev interpolation procedure:
\[
\boxed{
\text{equispaced physical nodes}
\quad\xlongrightarrow{\;S\;}\quad
\text{Padua points},
}
\]
while the sampled function values remain unchanged.

As a test problem, one may consider the {\it bivariate Runge function}
\[ f(x,y)=\frac{1}{1+25(x^2+y^2)},
\qquad (x,y)\in[-1,1]^2.
\]

The classical total-degree polynomial interpolant constructed on $X_n$
can then be compared with the mapped interpolant
\[
\mathcal I_n^S f(x,y)
=
p_n[g](S(x,y)).
\]
The purpose of the map is to preserve the original physical samples while
transferring the interpolation geometry to the more stable Padua
configuration. See figure\ref{fig:placeholder} and the following table for the results
\begin{table}[htbp]
\centering

\label{tab:fna-padua-runge}
\begin{tabular}{llll}
\hline
\textbf{Method}
& \textbf{Maximum error}
& \textbf{RMS error}
& $\boldsymbol{\kappa_2(V)}$ \\
\hline
Fake Nodes--Padua
& $2.03\times 10^{-2}$
& $2.03\times 10^{-3}$
& $2.42$ \\
Classical equispaced interpolation
& $1.33\times 10^{3}$
& $2.10\times 10^{1}$
& $7.32\times 10^{4}$ \\
\hline
\end{tabular}
\caption{Comparison between the Fake Nodes Approach based on Padua points 
and classical polynomial interpolation for the two-dimensional Runge function,
with degree $n=20$ and $N=231$ interpolation nodes.}
\end{table}

\begin{figure}[h!]
    \centering
    \includegraphics[width=0.85\linewidth]{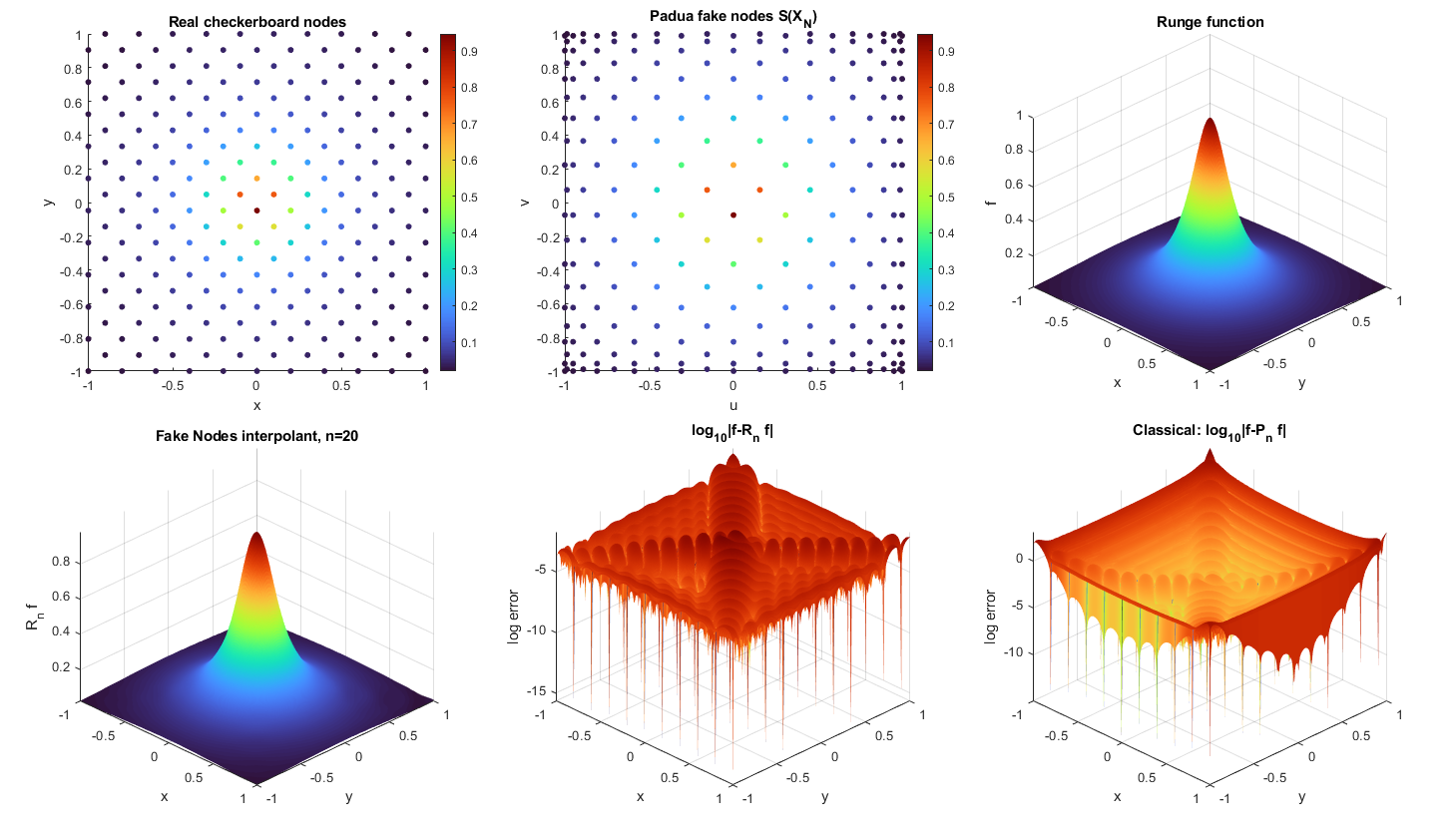}
    \caption{The Runge function reconstructed using the FNA}
    \label{fig:placeholder}
\end{figure}
\vskip 0.1in
The experiments reveal complementary roles for the mappings. Here we add some comments.
\begin{itemize}
\item Rescaling is inexpensive because it reuses the same matrix factorization and adds only the interpolant of the constant function. It is particularly natural for localized bases and partition-of-unity methods.
\item VSK changes pairwise distances and can encode known shape information. Its success depends on constructing a reliable scale function; noisy pilot approximants may degrade both accuracy and conditioning.
\item VSDK uses discontinuous labels to decouple regions separated by interfaces. It is therefore suited to image reconstruction and piecewise-smooth data.
\item Fake Nodes Approach modify a basis without resampling, allowing stable node distributions to be exploited even when measurements are acquired elsewhere.
\item Rational kernels introduce a nonlinear trial manifold. They include rescaling as a special case and offer additional shape control, but denominator safety is indispensable.
\end{itemize}

Our unifying perspective also suggests hybrid constructions. A mapped VSK combines a domain map $S$ and a scale function $\psi$,
\[
 K_{\psi,S}(x,y)=K((S(x),\psi(x)),(S(y),\psi(y))),
\]
while rational VSK expansions combine geometric lifting with a quotient trial space. These combinations are promising in inverse problems and scientific imaging, where geometry, discontinuities, and dynamic range must be treated simultaneously, see e.g. \cite{Marchetti:solar}.

\section{Conclusions}
We have organized rescaled interpolation, variably scaled kernels, fake nodes, and rational kernel expansions around a single principle: improve approximation by mapping the representation while preserving the original sampled values. The methods act at different levels, yet share the same interpolation invariant. This viewpoint clarifies their relationships and provides a systematic route to hybrid methods.

Future work should focus on data-driven maps with stability constraints, automatic detection of discontinuities, rational denominator selection, local and multilevel implementations, and learned maps that retain the mathematical guarantees of positive definiteness and interpolation. The central challenge is no longer merely to choose a kernel or a node set, but to design a map that exposes the structure of the data without amplifying noise or ill-conditioning.

\vskip 0.2in
{\bf Acknowldgments}. The author gratefully acknowledges the support and encouragement of the
\emph{Unione Matematica Italiana} (UMI), in particular through the
Thematic Group on Approximation Theory and Applications.
The author also acknowledges the support of the
\emph{Gruppo Nazionale per il Calcolo Scientifico} (GNCS),
which is part of the \emph{Istituto Nazionale di Alta Matematica} (INdAM).

\end{document}